\documentclass{amsart}
\usepackage{mathrsfs,latexsym,amsfonts,amssymb}

\newtheorem{theorem}{Theorem}[section]
\newtheorem{lemma}[theorem]{Lemma}
\newtheorem{corollary}[theorem]{Corollary}
\newtheorem{question}[theorem]{Question}
\theoremstyle{definition}
\newtheorem{definition}[theorem]{Definition}

\theoremstyle{remark}

\begin{document}

\title[On rectifiable spaces and paratopological groups]
{On rectifiable spaces and paratopological groups}

\author{Fucai Lin*}
\address{Fucai Lin(corresponding author): Department of Mathematics and Information Science,
Zhangzhou Normal University, Zhangzhou 363000, P. R. China}
\email{linfucai2008@yahoo.com.cn}

\author{Rongxin Shen}
\address{Rongxin Shen: Department
of Mathematics, Taizhou Teacher's College, Taizhou 225300, P. R.
China}
\email{srx20212021@163.com}

\thanks{*Supported by the NSFC (No. 10971185) and the Educational Department of Fujian Province (No. JA09166) of China.}


\begin{abstract}
We mainly discuss the cardinal invariants and generalized metric
properties on paratopological groups or rectifiable spaces, and show
that: (1) If $A$ and $B$ are $\omega$-narrow subsets of a
paratopological group $G$, then $AB$ is $\omega$-narrow in $G$,
which give an affirmative answer for \cite[Open problem
5.1.9]{A2008}; (2) Every bisequential or weakly first-countable
rectifiable space is metrizable; (3) The properties of
Fr$\acute{e}$chet-Urysohn and strongly Fr$\acute{e}$chet-Urysohn are
coincide in rectifiable spaces; (4) Every rectifiable space $G$
contains a (closed) copy of $S_{\omega}$ if and only if $G$ has a
(closed) copy of $S_{2}$; (5) If a rectifiable space $G$ has a
$\sigma$-point-discrete closed $k$-network, then $G$ contains no closed
copy of $S_{\omega_{1}}$; (6) If a rectifiable space $G$ is
pointwise canonically weakly pseudocompact, then $G$ is a Moscow
space. Also, we consider the remainders of paratopological groups or
rectifiable spaces, and give a partial answer to questions posed by C. Liu in
\cite{Liu2009} and C. Liu, S. Lin in \cite{Liu20091}, respectively.
\end{abstract}

\maketitle

\section{Introduction}
Recall that a {\it topological group} $G$ is a group $G$ with a
(Hausdorff) topology such that the product maps of $G \times G$ into
$G$ is jointly continuous and the inverse map of $G$ onto itself
associating $x^{-1}$ with arbitrary $x\in G$ is continuous. A {\it
paratopological group} $G$ is a group $G$ with a topology such that
the product maps of $G \times G$ into $G$ is jointly continuous. A
topological space $G$ is said to be a {\it rectifiable space}
provided that there are a surjective homeomorphism $\varphi :G\times
G\rightarrow G\times G$ and an element $e\in G$ such that
$\pi_{1}\circ \varphi =\pi_{1}$ and for every $x\in G$ we have
$\varphi (x, x)=(x, e)$, where $\pi_{1}: G\times G\rightarrow G$ is
the projection to the first coordinate. If $G$ is a rectifiable
space, then $\varphi$ is called a {\it rectification} on $G$. It is
well known that rectifiable spaces and paratopological groups are
all good generalizations of topological groups. In fact, for a
topological group with the neutral element $e$, then it is easy to
see that the map $\varphi (x, y)=(x, x^{-1}y)$ is a rectification on
$G$. However, there exists a paratopological group which is not a
rectifiable space; Sorgenfrey line (\cite[Example
1.2.2]{E1989}) is such an example. Also, the 7-dimensional sphere $S_{7}$ is
rectifiable but not a topological group \cite[$\S$ 3]{V1990}.
Further, it is easy to see that paratopological groups and
rectifiable spaces are all homogeneous.

By a remainder of a space $X$ we understand the subspace
$bX\setminus X$ of a Hausdorff compactification $bX$ of $X$.

In this article, we mainly consider the cardinal invariants and
generalized metric properties of paratopological groups or
rectifiable spaces. We also consider the question: When does a
Tychonoff paratopological group or rectifiable space $X$ have a
Hausdorff compactification $bX$ with a remainder belonging to a
given class of spaces?

\maketitle

\section{Preliminaries}
Let $X$ be a space. A collection of nonempty open sets $\mathscr{U}$
of $X$ is called a {\it $\pi$-base at point $x$} if for every
nonempty open set $O$ with $x\in O$, there exists an
$U\in\mathscr{U}$ such that $U\subset O$. The {\it $\pi$-character}
of $x$ in $X$ is defined by $\pi\chi(x,
X)=\mbox{min}\{|\mathscr{U}|:\mathscr{U}\ \mbox{is a}\
\pi\mbox{-base at point}\ x\ \mbox{in}\ X\}$. The {\it
$\pi$-character of $X$} is defined by
$\pi\chi(X)=\mbox{sup}\{\pi\chi(x, X):x\in X\}$.

\begin{definition}\cite{A2008}
Let $\zeta$ and $\eta$ be any family of non-empty subsets of $X$.
\begin{enumerate}
\item The family $\zeta$ is called a {\it prefilter} on a space $X$ if, whenever $P_{1}$
and $P_{2}$ are in $\zeta$, there is a $P\in\zeta$ such that
$P\subset P_{1}\cap P_{2}$;

\item A prefilter $\zeta$ on a space $X$ is
said to {\it converge to a point} $x\in X$ if every open
neighborhood of $x$ contains an element of $\zeta$;

\item If $x\in X$
belongs to the closure of every element of a prefilter $\zeta$ on
$X$, we say that $\zeta$ {\it accumulates to} $x$ or {\it a cluster
point for} $\zeta$;

\item Two prefilters $\zeta$ and $\eta$ are called to
be {\it synchronous} if, for any $P\in\zeta$ and any $Q\in\eta$,
$P\cap Q\neq\emptyset$;

\item A space $X$ is called {\it bisequential}
\cite{A2008} if, for every prefilter $\zeta$ on $X$ accumulating to
a point $x\in X$, there exists a countable prefilter $\xi$ on $X$
converging to the same point $x$ such that $\zeta$ and $\xi$ are
synchronous.
\end{enumerate}
\end{definition}

\begin{definition}
Let $\mathscr{P}=\bigcup_{x\in X}\mathscr{P}_{x}$ be a cover of a
space $X$ such that for each $x\in X$, (a) if $U,V\in
\mathscr{P}_{x}$, then $W\subset U\cap V$ for some $W\in
\mathscr{P}_{x}$; (b) the family $\mathscr{P}_{x}$ is a network of $x$ in $X$,
i.e., $x\in\bigcap\mathscr{P}_x$, and if $x\in U$ with $U$ open in
$X$, then $P\subset U$ for some $P\in\mathscr P_x$.

The family $\mathscr{P}$ is called a {\it weak base} for $X$ \cite{A1966} if, for every $G\subset X$, the set $G$ must be open in $X$ whenever for each $x\in G$ there exists $P\in
\mathscr{P}_{x}$ such that $P\subset G$.
The space $X$ is {\it weakly first-countable} if $\mathscr{P}_{x}$ is countable for each
$x\in X$.
\end{definition}

\begin{definition}Let $\kappa$ be an infinite cardinal.

\begin{enumerate}
\item A space $X$ is called an {\it
$S_{\kappa}$}-{space} if $X$ is obtained by identifying all the limit
points of $\kappa$ many convergent sequences;

\item A space $X$ is called an {\it $S_{2}$}-{space} ({\it Arens' space})  if
$X=\{\infty\}\cup \{x_{n}: n\in \mathbb{N}\}\cup\{x_{n}(m): m, n\in
\mathbb{N}\}$ and the topology is defined as follows: Each
$x_{n}(m)$ is isolated; a basic neighborhood of $x_{n}$ is
$\{x_{n}\}\cup\{x_{n}(m): m>k, \mbox{for some}\ k\in \mathbb{N}\}$;
a basic neighborhood of $\infty$ is $\{\infty\}\cup (\bigcup\{V_{n}:
n>k\ \mbox{for some}\ k\in \mathbb{N}\})$, where $V_{n}$ is a
neighborhood of $x_{n}$.
\end{enumerate}
\end{definition}

\begin{theorem}\cite{C1992, G1996, V1989}\label{t9}
A topological space $G$ is rectifiable if and only if there exists $e\in G$ and two
continuous maps $p: G^{2}\rightarrow G$, $q: G^{2}\rightarrow G$
such that for any $x\in G, y\in G$ the next
identities hold:
$$p(x, q(x, y))=q(x, p(x, y))=y, q(x, x)=e.$$
\end{theorem}

In fact, we can assume that $p=\pi_{2}\circ \varphi^{-1}$ and
$q=\pi_{2}\circ \varphi$ in Theorem~\ref{t9}. Fixed a point $x\in
G$, then $f_{x}, g_{x}: G\rightarrow G$ defined with $f_{x}(y)=p(x,
y)$ and $g_{x}(y)=q(x, y)$, for each $y\in G$, are homeomorphism,
respectively. We denote $f_{x}, g_{x}$ with $p(x, G), q(x, G)$,
respectively.

Let $G$ be a rectifiable space, and let $p$ be the multiplication on
$G$. Further, we sometime write $x\cdot y$ instead of $p(x, y)$ and
$A\cdot B$ instead of $p(A, B)$ for any $A, B\subset G$. Therefore,
$q(x, y)$ is an element such that $x\cdot q(x, y)=y$; since $x\cdot
e=x\cdot q(x, x)=x$ and $x\cdot q(x, e)=e$, it follows that $e$ is a right neutral
element for $G$ and $q(x, e)$ is a right inverse for $x$. Hence a
rectifiable space $G$ is a topological algebraic system with
operation $p, q$, 0-ary operation $e$ and identities as above. It is
easy to see that this algebraic system need not to satisfy
the associative law about the multiplication operation $p$. Clearly,
every topological loop is rectifiable.

All spaces are $T_1$ and regular unless stated otherwise.
The notation $\mathbb{N}$ denotes the set of all positive natural numbers. The letter $e$
denotes the neutral element of a group and the right neutral element
of a rectifiable space, respectively. Readers may refer to
\cite{A2008, E1989, Gr1984} for notations and terminology not
explicitly given here.
\bigskip

\section{Cardinal invariants in paratopological groups or rectifiable spaces}
A subset $B$ of a paratopological group $G$ is called {\it
$\omega$-narrow} in $G$ if, for each neighborhood $U$ of the
identity in $G$, there is a countable subset $F$ of $G$ such that
$B\subset FU\cap UF$.

\begin{question}\cite[Open problem 5.1.9]{A2008}
Let $A$ and $B$ be $\omega$-narrow subsets of a paratopological
group $G$. Is the set $AB$ necessarily $\omega$-narrow in $G$?
\end{question}

Now we give an affirmative answer to this question.

\begin{theorem}
Let $A$ and $B$ be $\omega$-narrow subsets of a paratopological
group $G$. Then $AB$ is $\omega$-narrow in $G$.
\end{theorem}

\begin{proof}
Let $U$ be a neighborhood of the neutral element $e$ in $G$. Since $G$ is a
paratopological group, we can choose an open neighborhood $V$ of $e$
in $G$ with $V^{2}\subset U$. Since $B$ is $\omega$-narrow, there is
a subset $C$ of $G$ with $|C|\leq\omega$ satisfying $B\subset CV$.
For every $y\in C$, we can choose a neighborhood $W_{y}$ of $e$ in
$G$ such that $y^{-1}W_{y}y\subset V$. Since $A$ is $\omega$-narrow
in $G$, for every $y\in C$, there is a subset $K_{y}$ of $G$ with
$|K_{y}|\leq\omega$ and $A\subset K_{y}W_{y}$. Let $K=\bigcup_{y\in
C}K_{y}$ and $M=KC$. Obviously, $|M|\leq\omega$. Now we show that
$AB\subset MU$. Assume that $a\in A$ and $b\in B$. Then there is a
point $y\in C$ such that $b\in yV$. Therefore, there exists a point
$x\in K_{y}$ with $a\in xW_{y}$. Hence we have $$ab\in
xW_{y}yV=xy(y^{-1}W_{y}y)V\subset xyVV\subset xyU,$$that is, $ab\in
MU$. Thus we proved that $AB\subset MU$. Similarly, we can prove that
there exists a subset $P$ of $G$ with $|P|\leq\omega$ and $AB\subset
UP$. Put $D=M\cup P$. Obviously, we have $AB\subset DU\cap UD$ and
$|D|\leq\omega$. Therefore, the product $AB$ is $\omega$-narrow in
$G$.
\end{proof}

It is well known that a bisequential or weakly first-countable
topological group is metrizable. Now, we show that a bisequential or
weakly first-countable rectifiable space is also metrizable.

\begin{theorem}
Every bisequential rectifiable space $G$ is metrizable.
\end{theorem}

\begin{proof}
Since $G$ is a bisequential space, there exists a countable open
prefilter $\gamma$ on $G$ converging to some point $g\in G$
\cite[Lemma 4.7.11]{A2008}. Since $G$ is homogeneous, without loss of
generality, we can assume that $g=e$. Then the family $\{q(B, B):
B\in\gamma\}$ is a base at $e$. Indeed, for any open neighborhood
$U$ of $e$, we can find an open neighborhood $V$ of $e$ and
$B\in\gamma$ such that $q(V, V)\subset U$ and $B\subset V$. So we
have $$e\in q(B, B)\subset q(V, V)\subset U.$$ Therefore, the space $G$
is first-countable at point $e$. Since $G$ is homogeneous, the space $G$ is
first-countable. Thus $G$ is metrizable by \cite[Theorem
3.2]{G1996}.
\end{proof}

\begin{lemma}\label{l2}
Suppose that $\{V_{n}(x): n\in \mathbb{N} ,x\in G\}$ is a weak base
in a rectifiable space. For each $x\in G$ and each $n\in
\mathbb{N}$, put $W_{n}(x)=x\cdot V_{n}(e)$. Then $\{W_{n}(x): n\in
\mathbb{N} ,x\in G\}$ is a weak base in $G$.
\end{lemma}

\begin{proof}
For each $x\in G$, let $f_{x}: G\rightarrow G$ with $f_{x}(y)=p(x,
y)$ for each $y\in G$. Therefore, the map $f_{x}$ is a homeomorphism of $G$
onto $G$ such that $f_{x}(e)=p(x, e)=p(x, q(x, x))=x$. Put
$W_{n}(x)=f_{x}(V_{n}(e))=x\cdot V_{n}(e)$. It follows from
\cite[Proposition 4.7.2]{A2008} that $\{W_{n}(x): n\in \mathbb{N}
,x\in G\}$ is a weak base in $G$.
\end{proof}

\begin{lemma}\label{l3}
Suppose that $\{V_{n}(x): n\in \mathbb{N} ,x\in G\}$ is a weak base
in a rectifiable space. For each $x\in G$ and each $n\in
\mathbb{N}$, put $W_{n}(x)=(x\cdot V_{n}(e))\cdot V_{n}(e)$. Then
$\{W_{n}(x): n\in \mathbb{N} ,x\in G\}$ is a weak base in $G$.
\end{lemma}

\begin{proof}
By Lemma~\ref{l2}, we can assume that $V_{n}(x)=x\cdot V_{n}(e)$,
for each $x\in G$ and each $n\in \mathbb{N}$. Since $p$ is a
continuous map from $G\times G$ onto $G$, it is easy to see that
$\{W_{n}(x): n\in \mathbb{N} ,x\in G\}$ is a weak base in $G$.
\end{proof}

\begin{lemma}\cite{A2008}\label{l4}
Suppose that $\{V_{n}(x): n\in \mathbb{N} ,x\in G\}$ and
$\{W_{n}(x): n\in \mathbb{N} ,x\in G\}$ are two weak bases in a
space $G$. Then, for each $g\in G$ and every $n\in\mathbb{N}$, there
exists a $k\in \mathbb{N}$ such that $W_{k}(g)\subset V_{n}(g)$.
\end{lemma}

\begin{theorem}\label{t8}
Every weakly first-countable rectifiable space $G$ is metrizable.
\end{theorem}

\begin{proof}
Let $\{V_{n}(x): n\in \mathbb{N} ,x\in G\}$ be a weak base in $G$.
For each $x\in G$ and $n\in \mathbb{N}$, we have
$V_{n}(x)=x\cdot V_{n}(e)$ by Lemma~\ref{l2}. Let $U_{n}=\{x\in
V_{n}(e): x\cdot V_{k}(e)\subset V_{n}(e)\ \mbox{for some}\
k\in\mathbb{N}\}$. Obviously, we have $e\in U_{n}\subset V_{n}(e)$ by
Lemma~\ref{l2} and~\ref{l4}. Next we show that $U_{n}$ is open in
$G$. Indeed, take any $y\in U_{n}$. Then $y\cdot V_{k}(e)\subset
V_{n}(e)$ for some $k\in\mathbb{N}$. By Lemma~\ref{l2},~\ref{l3}
and~\ref{l4}, it is easy to see that there exists an $m\in
\mathbb{N}$ such that $(y\cdot V_{m}(e))\cdot V_{m}(e)\subset y\cdot
V_{k}(e)$. Hence $(y\cdot V_{m}(e))\cdot V_{m}(e)\subset V_{n}(e)$,
which implies that $V_{m}(y)=y\cdot V_{m}(e)\subset U_{n}$.
Therefore, the set $U_{n}$ is open in $G$. Thus $\{U_{m}: m\in \mathbb{N}\}$
is a countable base of $G$ at $e$. Then $G$ is metrizable by
\cite[Theorem 3.2]{G1996}.
\end{proof}

A space $X$ is said to be {\it Fr$\acute{e}$chet-Urysohn} if, for
each $x\in \overline{A}\subset X$, there exists a sequence
$\{x_{n}\}$ such that $\{x_{n}\}$ converges to $x$ and $\{x_{n}:
n\in\mathbb{N}\}\subset A$. A space $X$ is said to be {\it strongly
Fr$\acute{e}$chet-Urysohn} if the following condition is satisfied

(SFU) For every $x\in X$ and each sequence $\eta =\{A_{n}: n\in
\mathbb{N}\}$ of subsets of $X$ such that $x\in\bigcap_{n\in
\mathbb{N}}\overline{A_{n}}$, there is a sequence $\zeta =\{a_{n}:
n\in \mathbb{N\}}$ in $X$ converging to $x$ and intersecting
infinitely many members of $\eta$.

Obviously, a strongly Fr$\acute{e}$chet-Urysohn space is
Fr$\acute{e}$chet-Urysohn. However, the space $S_{\omega}$ is
Fr$\acute{e}$chet-Urysohn and non-strongly
Fr$\acute{e}$chet-Urysohn.

Next, we show that the properties of Fr$\acute{e}$chet-Urysohn and
strongly Fr$\acute{e}$chet-Urysohn coincide in rectifiable
spaces.

\begin{theorem}\label{t4}
If a rectifiable space $G$ is Fr$\acute{e}$chet-Urysohn, then it is
strongly Fr$\acute{e}$chet-Urysohn.
\end{theorem}

\begin{proof}
We can assume that $G$ is non-discrete. It is enough to verify the
condition (SFU) for $x=e$ since $G$ is homogeneous. Suppose
$e\in\bigcap_{n\in\mathbb{N}}\overline{A_{n}}$, where each $A_{n}$
is a subset of $G$. Fix a sequence $\{a_{n}: n\in\mathbb{N}\}\subset
G\setminus \{e\}$ converging to $e$. For each $n\in \mathbb{N}$,
since $p(e, e)=e$, we can fix an open neighborhood $V_{n}$ of $e$
such that $a_{n}\not\in p(V_{n}, V_{n})$. Since $q(a_{n}, a_{n})=e$,
we can fix an open neighborhood $W_{n}$ of $e$ such that
$q(a_{n}\cdot W_{n}, a_{n})\subset V_{n}$ and $W_{n}\subset V_{n}$.
Moreover, for each $n\in\mathbb{N}$, since $e\in\overline{A_{n}}$,
we may assume that $A_{n}\subset W_{n}$. Put $C_{n}=a_{n}\cdot
A_{n}$, for each $n\in\mathbb{N}$.

Claim 1: We have $e\not\in\overline{C_{n}}$ and $a_{n}\in\overline{C_{n}}$
for each $n\in\mathbb{N}$.

Indeed, if $e\in\overline{C_{n}}$, then $W_{n}\cap a_{n}\cdot
A_{n}\neq\emptyset$. Then we can choose $w_{n}\in W_{n}$ and
$w^{\prime}_{n}\in A_{n}$ such that $w_{n}=a_{n}\cdot
w^{\prime}_{n}$. Therefore, we have $a_{n}=(a_{n}\cdot w^{\prime}_{n})\cdot
q(a_{n}\cdot w^{\prime}_{n}, a_{n})=p(a_{n}\cdot w^{\prime}_{n},
q(a_{n}\cdot w^{\prime}_{n}, a_{n}))\subset p(W_{n}, V_{n})\subset
p(V_{n}, V_{n})$, which is a contradiction. Hence
$e\not\in\overline{C_{n}}$ for each $n\in\mathbb{N}$. For each $g\in
G$ and $A\subset G$, since $p(g, \overline{A})=\overline{p(g, A)}$,
it is easy to see that $a_{n}\in\overline{C_{n}}$ for each
$n\in\mathbb{N}$.

Let $C=\bigcup_{n\in\mathbb{N}}$, and hence $e\in\overline{C}$.
Since $G$ is Fr$\acute{e}$chet-Urysohn, there exists a sequence
$\zeta =\{c_{n}: n\in\mathbb{N}\}$ in $C$ converging to $e$. By the
Claim 1, we have $e\not\in\overline{C_{n}}$, and hence the sequence
$\zeta$ must intersects $C_{n}$ for infinitely many values of $n$.
For every $n\in\mathbb{N}$, choose a $k_{n}\in\mathbb{N}$ such that
$c_{n}\in C_{k_{n}}$. Then $c_{n}=a_{k_{n}}\cdot a_{k_{n}}^{\prime}$
for each $n\in\mathbb{N}$, where $a_{k_{n}}^{\prime}\in A_{k_{n}}$.
Put
$$b_{n}=a_{k_{n}}^{\prime}=q(a_{k_{n}}, p(a_{k_{n}},
a_{k_{n}}^{\prime}))=q(a_{k_{n}}, a_{k_{n}}\cdot
a_{k_{n}}^{\prime})=q(a_{k_{n}}, c_{n})\rightarrow e\ \mbox{as}\
n\rightarrow\infty.$$ Therefore, the sequence $\{b_{n}:
n\in\mathbb{N}\}$ converges to $e$ and intersects infinitely many
$A_{n}$'s. Thus, the condition (SFU) is satisfied, and hence $G$ is
strongly Fr$\acute{e}$chet-Urysohn.
\end{proof}

\begin{theorem}
The product of a Fr$\acute{e}$chet-Urysohn rectifiable space $G$
with a first-countable space $M$ is Fr$\acute{e}$chet-Urysohn.
\end{theorem}

\begin{proof}
Take any subset $A$ of $G\times M$ and any point $(x, y)\in
\overline{A}$. Fix a decreasing countable base $\{U_{n}:
n\in\mathbb{N}\}$ of $M$ at the point $y$. Put
$B_{n}=\pi_{1}((G\times U_{n})\cap A)$. Clearly, we have
$x\in\overline{B_{n}}$ for each $n\in\mathbb{N}$. We also have
$B_{n+1}\subset B_{n}$, since $U_{n+1}\subset U_{n}$. It follows
from Theorem~\ref{t4} that there exists a sequence $\{b_{n}:
n\in\mathbb{N}\}\subset G$ converging to $x$ and intersecting
$B_{n}$ for infinitely many $n\in\mathbb{N}$. For each
$k\in\mathbb{N}$, there are $b_{k}\in B_{n_{k}}$ and $c_{k}\in
U_{n_{k}}$ such that $(b_{k}, c_{k})\in A$ and $n_{k}> k$. Then,
clearly, the sequence $\{(b_{k}, c_{k}): k\in\mathbb{N}\}$ converges
to the point $(x, y)$.
\end{proof}

\begin{theorem}
The following conditions are equivalent for a rectifiable space $G$:
\begin{enumerate}
\item every compact subspace of $G$ is first-countable;

\item every compact subspace of $G$ is metrizable.
\end{enumerate}
\end{theorem}

\begin{proof}
Obviously, we have $(2)\Rightarrow (1)$.

$(1)\Rightarrow (2)$. Let $F$ be a non-empty compact subset of $G$.
Consider the map $g: G\times G\rightarrow G$ defined by $g(x,
y)=q(x, y)$, for all $x, y\in G$. Clearly, the map $g$ is continuous, and
the image $F_{1}=g(F\times F)$ is a compact subset of $G$ which
contains $e$ of $G$. Since every compact subspace of $G$ is
first-countable, the compact subspace $F_{1}$ is first-countable, thus $\chi
(e, F_{1})\leq \mathbb{N}$. Denote by $f$ the restriction of $g$ to
$F\times F$. For each $x\in G$, since $q(x, G)$ are homeomorphism
and $q(x, x)=e$, it is easy to see that $\Delta_{F}=f^{-1}(e)$,
where $\Delta_{F}$ is the diagonal in $F\times F$. Since $f$ is a
closed map, it follows that $\chi (\Delta_{F}, F\times F)=\chi (e,
F_{1})\leq \mathbb{N}$. Therefore, the subspace $F$ is metrizable by
\cite{Gr1984}.
\end{proof}

\bigskip

\section{Generalized metrizable properties on rectifiable spaces}
In this section, we discuss the generalized metrizable properties on
rectifiable spaces.

\begin{theorem}\label{t7}
Let $G$ be a rectifiable space. Then $G$ contains a (closed) copy of
$S_{\omega}$ if and only if $G$ has a (closed) copy of $S_{2}$.
\end{theorem}

\begin{proof}
Sufficiency. Since $G$ is homogeneous, without loss of generality, we
can assume that $A=\{e\}\cup\{x_{n}: n\in
\mathbb{N}\}\cup\{x_{n}(m): m, n\in\mathbb{N}\}$ is a closed copy of
$S_{2}$, where, $x_{n}\rightarrow e$ as $n\rightarrow \infty$ and
$x_{n}(m)\rightarrow x_{n}$ as $m\rightarrow\infty$ for each $n\in
\mathbb{N}$. For each $m, n\in \mathbb{N}$, put $y_{n}(m)=q(x_{n},
x_{n}(m))$. Then, for each $n\in \mathbb{N}$, we have $y_{n}(m)=q(x_{n},
x_{n}(m))\rightarrow q(x_{n}, x_{n})=e$ as $m\rightarrow \infty$.
For every $n\in \mathbb{N}$, let $A_{n}=\{y_{n}(m): m\in
\mathbb{N}\}$.

Claim 2: For each $m\in \mathbb{N}$, the set $F=\{n: A_{m}\cap A_{n}\
\mbox{is infinite}\}$ is finite.

Indeed, if not, then there exists an $m\in \mathbb{N}$ such that $F=\{n:
A_{m}\cap A_{n}\ \mbox{is infinite}\}$ is infinite. Take distinct
$q(x_{n_{i}}, x_{n_{i}}(m_{i}))\in A_{m}\cap A_{n_{i}}$ for $n_{i}$
with $n_{i}< n_{i+1}$. Then $q(x_{n_{i}},
x_{n_{i}}(m_{i}))\rightarrow e$ as $i\rightarrow\infty$ since
$q(x_{n_{i}}, x_{n_{i}}(m_{i}))\in A_{m}$ for each $i\in
\mathbb{N}$. Since $x_{n_{i}}(m_{i})=p(x_{n_{i}}, q(x_{n_{i}},
x_{n_{i}}(m_{i})))\rightarrow q(e, e)=e$ as $i\rightarrow\infty$,
$x_{n_{i}}(m_{i})\rightarrow e$ as $i\rightarrow\infty$, which is a
contradiction.

By the Claim 2, without loss of generality, we can assume that
$A_{i}\cap A_{j}=\emptyset$ for distinct $i, j\in \mathbb{N}$. Put
$B=\{e\}\cup\{y_{n}(m): n, m\in \mathbb{N}\}$.

Claim 3: The subspace $B$ of $G$ is a closed copy of $S_{\omega}$.

First, the set $B$ is closed in $G$. Indeed, if not, then there exists a point
$x\in G\setminus B$ such that $x\in \overline{B}$. Obviously, we have $x\neq
e$, and hence $p(e, x)\neq e$ since $p(e, e)=e$ and $P(e, G)$ is a
homeomorphism. Since $A$ is closed, there is an open neighborhood
$V$ of $e$ such that $|p(V, x\cdot V)\cap A|\leq 1$ or $p(V, x\cdot
V)\cap A\subset \{x_{k}\}\cup\{x_{k}(m): m\in \mathbb{N}\}$ for some
$k\in \mathbb{N}$. Let $U$ be an open neighborhood of $e$ with
$U\cdot (x\cdot U)\subset V\cdot (x\cdot V)$ and $e\not\in x\cdot
U$. Then $x\cdot U$ contains an infinite subset $\{y_{n_{i}}(m_{i}):
i\in \mathbb{N}\}\subset B$, where $n_{i}\neq n_{j}$ for distinct
$i, j\in \mathbb{N}$. Since $x_{n}\rightarrow e$ as
$n\rightarrow\infty$, we can assume that $\{x_{n}: i\in
\mathbb{N}\}\subset U$. Therefore, for each $i\in \mathbb{N}$, we
have
$$x_{n_{i}}(m_{i})=p(x_{n_{i}}, q(x_{n_{i}},
x_{n_{i}}(m_{i})))\subset p(U, x\cdot U)\subset p(V, x\cdot V),$$
which implies that $\{x_{n_{i}}(m_{i}): i\in \mathbb{N}\}\subset
p(V, x\cdot V)$. This is a contradiction.

Let $f: \mathbb{N}\rightarrow\mathbb{N}$. Then $C=\cup\{y_{n}(m):
m\leq f(n), n\in \mathbb{N}\}$ does not have any cluster point.
Indeed, if not, then there exists a point $x\in\overline{C\setminus\{x\}}$.
Suppose that $V_{1}$ is an open neighborhood of $e$ with $|p(V_{1},
x\cdot V_{1})\cap\{x_{n}(m): m\leq f(n), n\in \mathbb{N}\}|\leq 1$,
and that $U_{1}$ is an open neighborhood of $e$ with $U_{1}\cdot
(x\cdot U_{1})\subset V_{1}\cdot (x\cdot V_{1})$. Then $x\cdot
U_{1}$ contains an infinite subset $\{y_{k_{i}}(l_{i}): i\in
\mathbb{N}\}$ of $C$. Since $x_{n}\rightarrow e$ as
$n\rightarrow\infty$, we can assume that $\{x_{k_{i}}: i\in
\mathbb{N}\}\subset U_{1}$. Therefore, for each $i\in \mathbb{N}$,
we have
$$x_{k_{i}}(l_{i})=p(x_{k_{i}}, q(x_{k_{i}},
x_{k_{i}}(l_{i})))\subset p(U_{1}, x\cdot U_{1})\subset p(V_{1},
x\cdot V_{1}),$$ which implies that $\{x_{k_{i}}(l_{i}): i\in
\mathbb{N}\}\subset p(V_{1}, x\cdot V_{1})$. This is a
contradiction.

Necessity. Let $A=\{e\}\cup\{y_{n}(m): m, n\in\mathbb{N}\}$ be a
closed copy of $S_{\omega}$, where for each $n\in \mathbb{N}$,
$y_{n}(m)\rightarrow e$ as $m\rightarrow\infty$. It is obvious that
there is a non-trivial sequence $\{x_{n}\}$ converging to $e$ as
$n\rightarrow\infty$, where $x_{n}\neq e$ for each $n\in
\mathbb{N}$. For each $n\in \mathbb{N}$, let $U_{n}$ be an open
neighborhood of $x_{n}$ with $\overline{U_{i}}\cap
\overline{U_{j}}=\emptyset$ for distinct $i, j\in \mathbb{N}$. Let
$x_{n}(m)=x_{n}\cdot y_{n}(m)$, for each $n, m\in \mathbb{N}$. For
each $n\in \mathbb{N}$, we have $x_{n}(m)\rightarrow x_{n}$ as
$m\rightarrow\infty$. Without loss of generality, we assume that
$\{x_{n}(m): m\in \mathbb{N}\}\subset U_{n}$. Put
$B=\{e\}\cup\{x_{n}: n\in \mathbb{N}\}\cup \{x_{n}(m): n, m\in
\mathbb{N}\}.$

Claim 4: The subspace $B$ of $G$ is a closed copy of $S_{2}$.

First, we show that $B$ is closed in $G$. Suppose not, there  is a
point $x\in G\setminus B$ such that $x\in \overline{B}$. It is easy
to see that $q(e, x)\neq e$. Since $A$ is closed, there is an open
neighborhood $V$ of $e$ such that $q(V, x\cdot V)\cap
(A\setminus\{q(e, x)\})=\emptyset$. Let $U$ be an open neighborhood
of $e$ with $q(U, x\cdot U)\subset q(V, x\cdot V)$ and $x\cdot
\overline{U}\cap\{x_{n}: n\in \mathbb{N}\}=\emptyset$. Clearly, for
each $n\in \mathbb{N}$, the set $x\cdot U\cap \{x_{n}(m): m\in \mathbb{N}\}$
is finite. Moreover, the set $x\cdot U$ contains infinitely many elements of
$\{x_{n}(m): n, m\in \mathbb{N}\}$, and we denote them by
$\{x_{n_{i}}(m_{i}): i\in \mathbb{N}\}$. Since $x_{n}\rightarrow e$
as $n\rightarrow\infty$, without loss of generality, we assume that
$\{x_{n}: n\in \mathbb{N}\}\subset U$. Therefore, for each $i\in N$,
we have
$$y_{n_{i}}(m_{i})=q(x_{n_{i}}, p(x_{n_{i}}, y_{n_{i}}(m_{i})))\subset q(U, x\cdot U)\subset q(V, x\cdot V),$$
which implies that $q(V, x\cdot V)$ contains infinitely many
$y_{n}(m)$'s.  This is a contradiction.

If $f: \mathbb{N}\rightarrow\mathbb{N}$, similarly as in the proof
of Claim 3, then $\cup\{x_{n}(m): m\leq f(n), n\geq k\ \mbox{for
some}\ k\in\mathbb{N}\}$ is closed in $G$. Hence $B$ is a closed
copy of $S_{2}$.
\end{proof}

\begin{definition}
 Let $\mathscr{P}$ be a family of subsets of
a space $X$.
\begin{enumerate}
\item The family $\mathscr{P}$ is called a {\it $wcs^{\ast}$-network} \cite{LT}
of $X$, if whenever sequence $\{x_{n}\}$ converges to $x\in U\in\tau
(X)$, there is a $P\in\mathscr{P}$ such that $P\subset U$ and for
each $n\in \mathbb{N}, x_{m_{n}}\in P$ for some $m_{n} > n$.

\item The family $\mathscr{P}$ is called a {\it $k$-network} \cite{PO} if
whenever $K$ is a compact subset of $X$ and $K\subset U\in \tau
(X)$, there is a finite subfamily $\mathscr{P}^{\prime}\subset
\mathscr{P}$ such that $K\subset \cup\mathscr{P}^{\prime}\subset U$.
\end{enumerate}
\end{definition}

\begin{corollary}\label{c0}
Let $G$ be a sequential rectifiable space. If $G$ has a
point-countable $wcs^{\ast}$-network, then $G$ is metrizable if and
only if $G$ contains no closed copy of $S_{2}$.
\end{corollary}

\begin{proof}
Obviously, it is sufficient to show the sufficiency. If $G$ contains no
closed copy of $S_{2}$, then $G$ contains no closed copy of
$S_{\omega}$ by Theorem~\ref{t7}, and hence $G$ is first-countable
space by \cite[Corollary 2.1.11]{Ls3}. Therefore, it follows that $G$ is metrizable
by \cite{G1996}.
\end{proof}

Let $\mathscr{B}=\{B_{\alpha}:\alpha\in H\}$ be a family of subsets
of a space $X$. The family $\mathscr{B}$ is {\it point-discrete} if
$\{x_{\alpha}:\alpha\in H\}$ is closed and discrete in $X$, whenever
$x_{\alpha}\in B_{\alpha}$ for each $\alpha\in H$. The family $\mathscr{B}$ is
{\it hereditarily closure-preserving} (abbrev. HCP) if, whenever a
subset $S(P)\subset P$ is chosen for each $P\in\mathscr{B}$, the
family $\{S(P): P\in \mathscr{B}\}$ is closure-preserving.

\begin{theorem}\label{t12}
Let $G$ be a rectifiable space. If $G$ has a $\sigma$-point-discrete
closed $k$-network, then $G$ contains no closed copy of $S_{\omega_{1}}$.
\end{theorem}

\begin{proof}
Suppose that $G$ contains a closed copy of
$S_{\omega_{1}}=\{e\}\cup\{x_{n}(\alpha): \alpha <\omega_{1}, n\in
\mathbb{N}\}$, where, for each $\alpha <\omega_{1}$,
$x_{n}(\alpha)\rightarrow e$ as $n\rightarrow\infty$. Obviously,
there exists a non-trivial sequence $\{x_{n}\}$ such that $x_{n}$
converges to $e$. By the regularity of $G$, we can take an open
subset $U_{n}$ of $G$ such that $x_{n}\in U_{n}$,
$\overline{U_{i}}\cap\overline{U_{j}}=\emptyset (i\neq j)$ and
$\overline{U_{n}}\cap \{x_{i}: i\in \mathbb{N}\}=\{x_{n}\}$. For
each $m\in \mathbb{N}$ and $\alpha <\omega_{1}$, it is easy to see
that $x_{m}\cdot x_{n}(\alpha)\rightarrow x_{m}$ as $n\rightarrow
\infty$ and $\{x_{m}\cdot x_{n}(\alpha): n\in \mathbb{N}\}$ is
eventually in $U_{m}$. Without loss of generality, we assume that
$\{x_{m}\cdot x_{n}(\alpha): n\in \mathbb{N}\}\subset U_{m}$.

Claim 5: The subspace $B=\{x_{n(\alpha)}\cdot x_{m(\alpha)}(\alpha): \alpha
<\omega_{1}\}$ of $G$ is discrete for $n(\alpha),
m(\alpha)\in \mathbb{N}$.

Case 1: The set $\{n(\alpha): \alpha <\omega_{1}\}$ is finite.

We denote $\{n(\alpha): \alpha <\omega_{1}\}$ by $\{l_{1}, \cdots
,l_{k}\}$. Since $\{x_{g(\alpha)}(\alpha): \alpha <\omega_{1}\}$ is
discrete for each $g: \omega_{1}\rightarrow\mathbb{N}$ and, for each
$x\in G$, the map $p(x, G)$ is a homeomorphism, the set $\{x_{l_{i}}\cdot
x_{g(\alpha)}(\alpha): \alpha <\omega_{1}\}$ is discrete for each
$1\leq i\leq k$. Therefore, the subspace $B$ is discrete.

Case 2: The set $\{n(\alpha): \alpha <\omega_{1}\}$ is infinite.

Suppose that $B$ is non-discrete, and that $x$ is a cluster point of
$B$. For each $g: \omega_{1}\rightarrow\mathbb{N}$, there exists an
open neighborhood $V$ of $e$ with $$|q(V, (x\cdot V))\cap
\{x_{g(\alpha)}(\alpha): \alpha <\omega_{1}\}|\leq 1.$$ Let $U$ be
an open neighborhood of $e$ with $q(U, (x\cdot U))\subset q(V,
(x\cdot V))$. Obviously, the set $C=x\cdot U\cap \{x_{n(\alpha)}\cdot
x_{m(\alpha)}(\alpha): \alpha <\omega_{1}\}\neq\emptyset$ for
infinitely many $n(\alpha)$, and we denote it by $\{n_{i}: i\in
\mathbb{N}\}$. Since $x_{n}\rightarrow e$ as $n\rightarrow\infty$,
without loss of generality, we assume that $\{x_{n}: n\in
\mathbb{N}\}\subset U$. Obviously, for each $i\in \mathbb{N}$, we
have
$$x_{m(\alpha)}(\alpha)=q(x_{n_{i}(\alpha)}, p(x_{n_{i}(\alpha)},
x_{m(\alpha)}(\alpha)))\subset q(U, x\cdot U)\subset q(V, x\cdot
V).$$ Then $|q(V, x\cdot V)\cap \{x_{g(\alpha)}(\alpha): \alpha
<\omega_{1}\}|\geq\omega$, which is a contradiction.

For $\alpha <\omega_{1}$, let $C_{\alpha}=\{e\}\cup\{x_{n}: n\in
\mathbb{N}\}\cup\{x_{n}\cdot x_{i}(\alpha): n\in \mathbb{N}, i\geq
f_{n}(\alpha)\}$. Obviously, we have $x_{n}\cdot
x_{j_{n}}(\alpha)\rightarrow e$ as $n\rightarrow\infty$, where
$j_{n}> f_{n}(\alpha)$. Since every infinitely subset of
$C_{\alpha}$ has a cluster point in $C_{\alpha}$, it follows that $C_{\alpha}$ is
countably compact. It is well known that every countably compact
space with a $\sigma$-point-discrete network has a countable
network, and hence $C_{\alpha}$ is compact.

Let $\mathscr{P}=\bigcup_{n\in \mathbb{N}}\mathscr{P}_{n}$ be a
$\sigma$-point-discrete $k$-network consisting of closed subsets of
$G$. Then there exists a finite subfamily
$\mathscr{P}^{\prime}\subset \mathscr{P}$ such that $C_{0}\subset
\cup\mathscr{P}^{\prime}$. Pick a $P_{0}\in\mathscr{P}^{\prime}$
such that $P_{0}$ contains a point $a_{0}=x_{n(0)}\cdot x_{m(0)}(0)$
and infinitely many $x_{n}$'s. Suppose that, for each $\alpha
<\beta$, there exists a $P_{\alpha}\in\mathscr{P}$ such that
$P_{\alpha}$ contains a point $a_{\alpha}=x_{n(\alpha)}\cdot
x_{m(\alpha)}(\alpha)$ and infinitely many $x_{n}$'s. Then we have
$C_{\beta}\subset G\setminus\{a_{\alpha}: \alpha <\beta\}$, which is
open in $G$ by the Claim 5. Therefore, there exists a finite
subfamily $\mathscr{P}^{\prime\prime}\subset \mathscr{P}$ such that
$C_{\beta}\subset \cup\mathscr{P}^{\prime\prime}\subset
G\setminus\{a_{\alpha}: \alpha <\beta\}$. Pick a
$P_{\beta}\in\mathscr{P}$ such that $P_{\beta}$ contains a point
$a_{\beta}=x_{n(\beta)}\cdot x_{m(\beta)}(\beta)$ and infinitely
many $x_{n}$'s. By induction, we can pick that $\{P_{\alpha}: \alpha
<\omega_{1}\}\subset \mathscr{P}$ such that $P_{\alpha}\neq
P_{\beta}$ whenever $\alpha\neq\beta$ and each $P_{\alpha}$ contains
infinitely many $x_{n}$'s. Therefore, there are uncountably many
$P_{\alpha}\in\mathscr{P}_{n}$ for some $n\in \mathbb{N}$. Since
$\mathscr{P}_{n}$ is point-discrete, there is a subsequence $L$ of
$\{x_{n}: n\in \mathbb{N}\}$ such that $L$ is discrete, which is a
contradiction.
\end{proof}

In \cite{H1992}, H.J. Junnila and Y.Z. Qiu have proved that a
space $X$ is an $\aleph$-space\footnote{A space $X$ is called an
{\it $\aleph$-space} if $X$ has a $\sigma$-locally finite
$k$-network.} if and only if $X$ has a $\sigma$-HCP $k$-network and
contains no closed copy of $S_{\omega_{1}}$. Therefore, we have the
following corollary by Theorem~\ref{t12}.

\begin{corollary}
A rectifiable space $G$ is an $\aleph$-space if and only if $G$ has
a $\sigma$-HCP $k$-network.
\end{corollary}

\bigskip

\section{When is a rectifiable space a Moscow space?}
A space $X$ is called {\it Moscow} \cite{A2008} if for each open
subset $U$ of $X$, the set $\overline{U}$ is the union of a family of
$G_{\delta}$-sets in $X$, that is, for every $x\in \overline{U}$,
there exists a $G_{\delta}$-set $P$ in $X$ with $x\in P\subset
\overline{U}$.

A point $x\in X$ is said to be a {\it point of canonical weak
pseudocompactness} \cite{A2008} or a {\it cwp-point} of $X$, for
brevity, if the following condition is satisfied:

(CWP) For every regular open subset $U$ of $X$ such that $x\in
\overline{U}$, there is a sequence $\{A_{n}: n\in \mathbb{N}\}$ of
subsets of $U$ such that $x\in \overline{A_{n}}$, for any $n\in
\mathbb{N}$, and for every indexed family $\eta =\{O_{n}: n\in
\mathbb{N}\}$ of open subsets of $X$ satisfying $O_{n}\cap
A_{n}\neq\emptyset$ for any $n\in \mathbb{N}$, the family $\eta$ has
an accumulation point in $X$.

A space $X$ is called {\it pointwise canonically weakly
pseudocompact} \cite{A2008} if each point of $X$ is a cwp-point.

\begin{theorem}\label{t5}
If a rectifiable space $G$ is pointwise canonically weakly
pseudocompact, then $G$ is a Moscow space.
\end{theorem}

\begin{proof}
Let $U$ be a regular open subset of $G$. Obviously, it suffices to
show that if $e\in \overline{U}$, then there is a $G_{\delta}$-set
$P\subset G$ such that $e\in P\subset \overline{U}$. Thus let us
assume that $e\in \overline{U}$ and fix subsets $A_{n}\subset U$
such as in condition (CWP), where $x=e$.

Next we are going to define a sequence $\{V_{n}: n\in \mathbb{N}\}$
of open neighborhood $e$, and a sequence $\{a_{n}: n\in
\mathbb{N}\}\subset U$ with $a_{n}\in A_{n}$, for each $n\in
\mathbb{N}$. Firstly, take an $a_{1}\in A_{0}$, and let $V_{1}$ be
an open neighborhood of $e$ with $a_{1}\cdot V_{1}\subset U$. Assume
that an open neighborhood $V_{k}$ of $e$ is already defined, for
some $k\in \mathbb{N}$. Choose a point $a_{k+1}\in A_{k+1}\cap
V_{k}$. Let $V_{k+1}$ be an open neighborhood of $e$ such that
$\overline{V_{k+1}}\cdot V_{k+1}\subset V_{k}$, $q(V_{k+1},
V_{k+1})\subset V_{k}$ and $a_{k+1}\cdot V_{k+1}\subset U$. The
recursive definition is complete. Let $\zeta =\{a_{n}\cdot V_{n+1}:
n\in \mathbb{N}\}$ and $H$ be the set of all accumulation point of
$\zeta$ in $G$. Since $G$ is pointwise canonically weakly
pseudocompact, it follows that $H\neq\emptyset$. Put $B=\bigcap_{n\in
\mathbb{N}}V_{n}$. Obviously, we have $B=\bigcap_{n\in
\mathbb{N}}\overline{V_{n}}$, and hence $B$ is also closed in $G$.

Claim 6: The set $H$ contains in $B$.

It follows from $a_{n}\in V_{n-1}$ that we have $a_{n}\cdot
V_{n+1}\subset V_{n-1}\cdot V_{n+1}\subset V_{n-1}\cdot
V_{n-1}\subset V_{n-2}$, for each $n\geq 3$. Therefore, we have
$H\subset \overline{\cup\{a_{n}\cdot V_{n+1}: n\in \mathbb{N}, n\geq
k+1\}}\subset \overline{V_{k-1}}$ for each $k\geq 2$. Thus $H\subset
B$, whence Claim 6 follows.

Claim 7: We have $a\cdot B=B$, for every $a\in H$.

In fact, fix a point $a\in H$, we have $a\in B$ by Claim 6, and
hence, the point $a\in V_{n}$, for each $n\in \mathbb{N}$. For each $b\in B$
and $k\in \mathbb{N}$, since $p(a, b)\in V_{k}$, we have $p(a, b)\in
B$. Then $a\cdot B\subset B$. Take any point $b\in B$. For each
$n\in \mathbb{N}$, since $q(a, b)\in V_{n}$, it follows that $q(a, b)\in B$.
Therefore, we have $b=p(a, q(a, b))=a\cdot q(a, b)\in a\cdot B$. Thus
$B\subset a\cdot B$, and hence $a\cdot B=B$.

Fix a point $a\in H$. Then, by Claim 7, we have $$B=a\cdot B\subset
\overline{\bigcup_{n\in \mathbb{N}}(a_{n}\cdot V_{n+1})\cdot
B}\subset\overline{\bigcup_{n\in \mathbb{N}}(a_{n}\cdot V_{n})\cdot
V_{n}}\subset \overline{U}.$$ Since $e\in B$, it follows that $e\in
B\subset\overline{\bigcup_{n\in \mathbb{N}}(a_{n}\cdot V_{n})\cdot
V_{n}}\subset \overline{U}$. Since $B$ is a $G_{\delta}$-set, the space $G$ is
Moscow.
\end{proof}

\begin{lemma}\cite[Lemma 5]{V1990}\label{l12}
Let $X$ be a Moscow space, and suppose that $Y$ is a
$G_{\delta}$-dense\footnote{A subset $Y$ of $X$ is called {\it
$G_{\delta}$-dense} if every $G_{\delta}$ of $X$ meets $Y$.}
subspace of $X$. Then $Y$ is $C$-embedded\footnote{Remember that $Y$
is {\it C-embedded} in $X$ if every continuous real-valued function
on $Y$ extends continuously over $X$.} in $X$.
\end{lemma}

It follows from Theorem~\ref{t5} and Lemma~\ref{l12} that we have
the following corollary.

\begin{corollary}
Let $G$ be a pointwise canonically weakly pseudocompact rectifiable
space, and $Y$ a $G_{\delta}$-dense subspace of $G$. Then $Y$ is
$C$-embedded in $G$.
\end{corollary}

Let $G$ be a rectifiable space, and $U\subset G$. A subset $A$ of
$G$ is called {\it $\omega$-deep} in $U$ if there is a
$G_{\delta}$-set $B$ in $G$ with $e\in B$ and $A\cdot B\subset U$.
We say that the {\it $g$-tightness} $t_{g}(G)$ of $G$ is countable
if, for each regular open subset $U$ of $G$ and each $x\in
\overline{U}$, there is an $\omega$-deep subset $A$ of $U$ such that
$x\in \overline{A}$.

\begin{theorem}\label{t6}
Every rectifiable space of countable $g$-tightness is a Moscow
space.
\end{theorem}

\begin{proof}
Take any regular open subset $U$ of $G$, and any point
$x\in\overline{U}$. Since $t_{g}(G)\leq\omega$, there is an
$\omega$-deep subset $A$ of $U$ with $x\in \overline{A}$. Then we
can fix a $G_{\delta}$-subset $B$ of $G$ with $e\in B$ and $A\cdot
B\subset U$. For each $a\in G$, since $p(a, G)$ is a homeomorphism
map, we have $x\in x\cdot B\subset \overline{A\cdot B}\subset
\overline{U}$, and  $x\cdot B$ is a $G_{\delta}$-subset of $G$.
Thus, the space $G$ is Moscow.
\end{proof}

The following lemma is an easy exercise.

\begin{lemma}\label{l5}
The union of any countable family of $\omega$-deep subsets of $U$ is
an $\omega$-deep subset of $U$, for any set $U$ of a rectifiable
space $G$.
\end{lemma}

\begin{lemma}\label{l6}
If $G$ is a rectifiable space of countable tightness, then the
$g$-tightness of $G$ is countable.
\end{lemma}

\begin{proof}
It is easy to see by Lemma~\ref{l5}.
\end{proof}

\begin{lemma}\label{l7}
If $G$ is a rectifiable space of countable $o$-tightness\footnote{A
space $X$ is called has {\it countable o-tightness} if whenever a
point $a\in X$ belongs to the closure of $\cup\eta$, where $\eta$ is
any family of open subsets in $X$, there exists a countable
subfamily $\zeta$ of $\eta$ such that $a\in\overline{\cup\zeta}$.},
then the $g$-tightness of $G$ is countable. In particular, if $G$
has a countable cellularity, then $g$-tightness of $G$ is also
countable.
\end{lemma}

\begin{proof}
Let $U$ be a regular open subset of $G$, and assume that $x\in
\overline{U}$. Denote by $\zeta$ the family of all non-empty
$\omega$-deep open subsets of $U$. We have $U=\cup\zeta$, since
$G$ is a rectifiable space. Since the $o$-tightness of $G$ is
countable, there is a countable subfamily $\gamma\subset\zeta$ with
$x\in \overline{\cup\gamma}$. Then the family $\eta$ is countable,
and hence, the set $V=\cup\gamma$ is an $\omega$-deep subset of $U$
by Lemma~\ref{l5}. Hence, we have $t_{g}(G)\leq\omega$. Since the
$o$-tightness of a space $X$ is less than or equal to the
cellularity of $X$, whence the proof is complete.
\end{proof}

The next two lemmas are obvious.

\begin{lemma}\label{l8}
If $G$ is an extremelly disconnected rectifiable space\footnote{We
recall that a space $X$ is {\it extremally disconnected} if the
closure of any open subset of $X$ is open.}, then the $g$-tightness
of $G$ is countable.
\end{lemma}

\begin{lemma}\label{l9}
If $G$ is rectifiable space of countable pseudocharacter, then the
$g$-tightness of $G$ is countable.
\end{lemma}

\begin{theorem}\label{t11}
Every dense subspace of a rectifiable space of countable
$g$-tightness is a Moscow space. In particular, if a rectifiable
space $G$ satisfies at least one of the following conditions, then
it is Moscow.
\begin{enumerate}
\item the space $G$ is a dense subspace of a
$\kappa$-Fr$\acute{e}$chet-Urysohn\footnote{A space is called {\it
$\kappa$-Fr$\acute{e}$chet-Urysohn} if for any $x\in \overline{U}$
with $U$ open in $X$, there exists some sequence of points of $U$
converges to $x$.}rectifiable space;

\item the
tightness of $G$ is countable;

\item the $o$-tightness of $G$ is countable;

\item the cellularity of $G$ is countable;

\item the pseudocharacter of $G$ is countable;

\item the space $G$ is extremally disconnected.
\end{enumerate}
\end{theorem}

\begin{proof}
It follows from Lemma~\ref{l5} and~\ref{l6} that (1) holds. It
is easy to see that (3), (4), (5) and (6) follow from
Theorem~\ref{t6} and Lemma~\ref{l7},~\ref{l8} and~\ref{l9}.
\end{proof}

\begin{lemma}\label{l13}
Let $G$ be a rectifiable space, $U$ a subset of $G$, and $b$ an
element of $G$. If there exists a countable $\omega$-deep subsets
$\gamma$ of $U$ such that $b\in\overline{\cup\gamma}$, then there is
a closed $G_{\delta}$-subset $P$ of $G$ such that $b\in P\subset
\overline{U}$.
\end{lemma}

\begin{proof}
Without loss of generality, we may assume that $b=e$. Also, we
denote $\gamma$ by $\{F_{n}: n\in\mathbb{N}\}$. For each $n\in
\mathbb{N}$, it is easy to see that we can choose a closed
$G_{\delta}$-subset $V_{n}$ such that $e\in V_{n}$, $F_{n}\cdot
V_{n}\subset U$, $V_{n+1}\cdot V_{n+1}\subset V_{n}$ and $q(V_{n+1},
V_{n+1})\subset V_{n}$. Now let $P=\cap\{V_{n}: n\in \mathbb{N}\}$.
Obviously, the set $P$ is a closed $G_{\delta}$-subset of $G$. By the proof
of Claim 7 in Theorem~\ref{t5}, we also have $P=e\cdot P$.
Since $e\in\overline{\cup\gamma}$, it follows that
$$P=e\cdot P\subset \overline{\cup\{F_{n}\cdot P: n\in \mathbb{N}\}}\subset
\overline{\cup\{F_{n}\cdot V_{n}: n\in \mathbb{N}\}}\subset
\overline{U}.$$
\end{proof}

By Lemma~\ref{l13}, we have the following theorem.

\begin{theorem}
Let $G$ be a rectifiable space. For each open subset $U$ and each
point $b\in \overline{U}$, if there exists a countable $\omega$-deep
subsets $\gamma$ of $U$ such that $b\in\overline{\cup\gamma}$, then
$G$ is a Moscow space.
\end{theorem}

\begin{theorem}
Let $G$ be a sequential rectifiable space such that
$G\setminus\{e\}$ is normal. Then $G$ has countable pseudocharacter.
\end{theorem}

\begin{proof}
We may assume that $e$ is a non-isolated point in $G$. Since $G$ is
homogeneous, we only need to show that $e$ is a $G_{\delta}$-point.
Suppose that $e$ is a non-$G_{\delta}$-point in $G$. Obviously,
we have $G\setminus\{e\}$ is a non-sequentially closed subset, and hence
there exists a sequence $\{x_{n}: n\in \mathbb{N}\}\subset
G\setminus\{e\}$ converging to $e$. Put $A=\{x_{2n-1}: n\in
\mathbb{N}\}$ and $B=\{x_{2n}: n\in \mathbb{N}\}$. Clearly, we may
assume that $A$ and $B$ are disjoint. Obviously, the sets $A$ and $B$ are
closed in $G\setminus\{e\}$, and since $G\setminus\{e\}$ is normal,
there exists a continuous function $f$ on $G\setminus\{e\}$ such
that $f(x)=1$, for each $x\in A$, and $f(x)=0$, for each $x\in B$.
Therefore, it is impossible to extend this function continuously to
the point $e$. However, since $G$ is sequential, the space $G$ is Moscow
by (2) in Theorem~\ref{t11}, and hence $G\setminus\{e\}$ is
$C$-embedded in $G$ by Lemma~\ref{l12}, which is a contradiction.
\end{proof}

\bigskip

\section{Remainders of $k$-gentle paratopological groups or rectifiable spaces}
In this section, we assume that all spaces are Tychonoff.

Let $f: X\rightarrow Y$ be a map. The map $f$ is called {\it k-gentle} if
for each compact subset $F$ of $X$ the image $f(F)$ is also compact.
A paratopological group $G$ is called {\it k-gentle} \cite{A2009} if
the inverse map $x\mapsto x^{-1}$ is $k$-gentle.

\begin{lemma}\cite{A2009}\label{l0}
Suppose that $G$ is a $k$-gentle paratopological group. Then any
remainder of $G$ in a compactification $bG$ of $G$ is either
pseudocompact or Lindel$\ddot{\mbox{o}}$f.
\end{lemma}

\begin{lemma}\cite{A2009}\label{l1}
Let $G$ be a $k$-gentle paratopological group such that some
remainder of $G$ is Lindel$\ddot{\mbox{o}}$f. Then $G$ is a
topological group.
\end{lemma}

\begin{theorem}({\bf Henriksen and Isbell} \cite{H1958})\label{t14}  A space
$X$ is of countable type if and only if its remainder in any (in
some) compactification of $X$ is Lindel\"{o}f.
\end{theorem}

\begin{theorem}\label{t3}
Suppose that $G$ is a paratopological group, and $Y=bG\setminus G$
is a remainder of $G$. If $Y$ has countable pseudocharacter, then at
least one of the following conditions is satisfied
\begin{enumerate}
\item the space $G$ is of countable type\footnote{Recall that a space $X$ is of {\it countable type} if every
compact subspace $F$ of $X$ is contained in a compact subspace
$K\subset X$ with a countable base of open neighborhoods in $X$.};

\item the space $Y$ is first countable.
\end{enumerate}
\end{theorem}

\begin{proof}
Let $Y$ be a non-first countable space. Then there exists a point
$y_{0}\in Y$ such that $Y$ is not first countable at point $y_{0}$.
Since $Y$ has countable pseudocharacter, the point $y_{0}$ is a
$G_{\delta}$-point in $Y$, and hence there exists a compact subset
$F\subset bG$ such that $F$ has a countable open neighborhood base
in $bG$ and $F\cap (bG\setminus G)=\{y_{0}\}$. Then
$F\setminus\{y_{0}\}\neq\emptyset$ because $Y$ is not first
countable at point $y_{0}$. Therefore, there is a non-empty compact
subset $B\subset F$ such that $B$ has a countable neighborhood base
in $bG$ and $y_{0}\not\in B$. It is obvious that $B\subset G$. It
follows from \cite[Proposition 4.1]{A2009} that $G$ is of countable
type.
\end{proof}

\begin{theorem}
Suppose that $G$ is a non-locally compact, $k$-gentle
paratopological group, and $Y=bG\setminus G$ is a remainder of $G$.
If $Y$ has locally a regular $G_{\delta}$-diagonal\footnote{A space
$X$ is said to have a {\it regular $G_{\delta}$-diagonal} if the
diagonal $\Delta=\{(x, x): x\in X\}$ can be represented as the
intersection of the closures of a countable family of open
neighborhoods of $\Delta$ in $X\times X$.}, then $G$ is a
topological group, and hence $G, bG$ and $Y$ are separable and
metrizable spaces.
\end{theorem}

\begin{proof}
Clearly, the space $Y$ is nowhere locally compact since $G$ is
non-locally compact. It follows from Lemma~\ref{l0} that $Y$ is
pseudocompact or Lindel$\ddot{\mbox{o}}$f.

Claim 8: The space $Y$ is Lindel$\ddot{\mbox{o}}$f.

Suppose not, we assume that $Y$ is pseudocompact. Since $Y$ has
locally a regular $G_{\delta}$-diagonal, for each $y\in Y$, there is
an open neighborhood $U_{y}$ of $y$ in $Y$ such that
$\overline{U_{y}}$ has a regular $G_{\delta}$-diagonal. Obviously,
the set $\overline{U_{y}}$ is pseudocompact because $Y$ is pseudocompact.
For each $y\in Y$, the set $\overline{U_{y}}$ is metrizable, and hence
$\overline{U_{y}}$ is compact. Then $Y$ is locally compact, which is
a contradiction.

It follows from Lemma~\ref{l1} and Claim 8 that $G$ is a topological
group. Therefore, it follows that $G, bG$ and $Y$ are separable and metrizable by
\cite{A2007}.
\end{proof}

\begin{corollary}
Suppose that $G$ is a non-locally compact, $k$-gentle
paratopological group, and $Y=bG\setminus G$ is a remainder of $G$.
If $Y$ has a regular $G_{\delta}$-diagonal, then $G, bG$ and $Y$ are
separable and metrizable spaces.
\end{corollary}

\begin{theorem}
Suppose that $G$ is a non-locally compact, $k$-gentle
paratopological group, and $Y=bG\setminus G$ is a remainder of $G$.
If $Y$ has locally a $\sigma$-point-finite base, then $G$ is a
topological group, and hence $G, bG$ and $Y$ are separable and
metrizable spaces.
\end{theorem}

\begin{proof}
It follows from Lemma~\ref{l0} that $Y$ is pseudocompact or
Lindel$\ddot{\mbox{o}}$f.

Suppose that $Y$ is pseudocompact. Since $Y$ has locally a
$\sigma$-point-finite base, for each $y\in Y$, there is an open
neighborhood $U_{y}$ of $y$ in $Y$ such that $\overline{U_{y}}$ has
a $\sigma$-point-finite base. Clearly, the set $\overline{U_{y}}$ is
pseudocompact because $Y$ is pseudocompact. For each $y\in Y$,
the set $\overline{U_{y}}$ is metrizable \cite{V1984}, and hence
$\overline{U_{y}}$ is compact. Then $Y$ is locally compact which is
a contradiction. Therefore, the space $Y$ is Lindel$\ddot{\mbox{o}}$f.

It follows from Lemma~\ref{l1} that $G$ is a topological group.
Therefore, we have $G, bG$ and $Y$ are separable and metrizable by
\cite{LFC2009}.
\end{proof}

\begin{corollary}
Suppose that $G$ is a non-locally compact, $k$-gentle
paratopological group, and $Y=bG\setminus G$ is a remainder of $G$.
If $Y$ has locally an uniform base\footnote{Let the family $\mathscr{P}$ be a
base of a space $X$. The family $\mathscr{P}$ is an {\it uniform
base} \cite{E1989} for $X$ if for each point $x\in X$ and
$\mathscr{P}^{\prime}$ is a countably infinite subset of
$(\mathscr{P})_{x}$, the family $\mathscr{P}^{\prime}$ is a neighborhood base
at $x$.}, then $G, bG$ and $Y$ are separable and metrizable spaces.
\end{corollary}

\begin{question}\label{q2}
Suppose that $G$ is a non-locally compact, $k$-gentle
paratopological group, and $Y=bG\setminus G$ is a remainder of $G$.
If $Y$ has a point-countable base, are $G, bG$ and $Y$ separable and
metrizable spaces?
\end{question}

\begin{question}\label{q3}
Suppose that $G$ is a non-locally compact, $k$-gentle
paratopological group, and $Y=bG\setminus G$ is a remainder of $G$.
If $Y$ has a $G_{\delta}$-diagonal, are $G, bG$ and $Y$ separable
and metrizable spaces?
\end{question}

Next, we give some partial answers to Questions~\ref{q2}
and~\ref{q3}.

\begin{theorem}\label{t0}
Suppose that $G$ is a non-locally compact, $k$-gentle
paratopological group, and $Y=bG\setminus G$ is a remainder of $G$.
If $Y$ has a point-countable base and a $G_{\delta}$-diagonal, then
$G, bG$ and $Y$ are separable and metrizable spaces.
\end{theorem}

\begin{proof}
Since $Y$ has a $G_{\delta}$-diagonal, the space $G$ is $\sigma$-compact or is
of countable type \cite{A2009}.

Case 1: The space $G$ is of countable type.

Then $Y$ is Lindel$\ddot{\mbox{o}}$f, and hence $G$ is a topological
group by Lemma~\ref{l1}. Therefore, it follows that $G, bG$ and $Y$ are separable
and metrizable spaces by \cite{A2007}.

Case 2: The space $G$ is $\sigma$-compact.

It is obvious that $c(G)\leq\omega$, and hence $c(bG)\leq\omega$.
Since $Y$ is dense in $bG$, we have $c(Y)\leq\omega$. There exists a
paracompact $\check{C}$ech-complete dense subset $Z\subset Y$
because $Y$ is $\check{C}$ech-complete. Since $Z$ has a
point-countable base, the subspace $Z$ is metrizable by \cite[Corollary
7.10]{Gr1984}. We have $c(Z)\leq\omega$ because $Z$ is dense in $Y$. Hence
$Z$ is separable, and $Y$ is a separable space. Therefore, the space $Y$ has a
countable base since $Y$ is separable space with a point-countable
base. Thus $Y$ is Lindel$\ddot{\mbox{o}}$f, and hence $G$ is a
topological group by Lemma~\ref{l1}. Therefore, it follows that $G, bG$ and $Y$ are
separable and metrizable spaces by \cite{A2007}.
\end{proof}

\begin{theorem}
Suppose that $G$ is a non-locally compact, $k$-gentle
paratopological group, and $Y=bG\setminus G$ is a remainder of $G$.
If $Y$ is locally normal with a $G_{\delta}$-diagonal, then $G, bG$
and $Y$ are separable and metrizable spaces.
\end{theorem}

\begin{proof}
It follows from Lemma~\ref{l0} that $Y$ is pseudocompact or
Lindel$\ddot{\mbox{o}}$f.

Case 1: The space $Y$ is pseudocompact.

For every $y\in Y$, since $Y$ is locally normal with a
$G_{\delta}$-diagonal, there exists an open neighborhood $U$ of $y$
such that $\overline{U}$ is is normal subspace with a
$G_{\delta}$-diagonal. Then $\overline{U}$ is pseudocompact, and
hence $\overline{U}$ is countably compact and metirzable. It follows
that $Y$ is locally compact. Therefore, it follows that $G, bG$ and $Y$ are
separable and metrizable spaces by Theorem~\ref{t0}.

Case 2: The space $Y$ is Lindel$\ddot{\mbox{o}}$f.

It follows from Lemma~\ref{l1} that $G$ is a topological group.
Since $Y$ has locally a $G_{\delta}$-diagonal, we have $G, bG$ and $Y$ are
separable and metrizable spaces \cite{Liu2009}.
\end{proof}

\begin{theorem}
Suppose that $G$ is a non-locally compact, $k$-gentle
paratopological group, and $Y=bG\setminus G$ is a remainder of $G$.
If $Y$ is a c.c.c space with a sharp base\footnote{Let the family $\mathscr{P}$
be a base of a space $X$. The family $\mathscr{P}$ is a {\it sharp
base} \cite{AJ} for $X$ if for each point $x\in X$ and $\{P_{n}:
n\in \mathbb{N}\}\subset (\mathscr{P})_{x}$, where $P_{n}\neq P_{m}$
whenever $n\neq m$, the family $\{\cap_{n\leq i}P_{n}:i\in\mathbb{N}\}$ is a
neighborhood base at $x$.}, then $G, bG$ and $Y$ are separable and
metrizable spaces.
\end{theorem}

\begin{proof}
It follows from Lemma~\ref{l0} that $Y$ is pseudocompact or
Lindel$\ddot{\mbox{o}}$f.

Case 1: The space $Y$ is pseudocompact.

Since a pseudocompact and c.c.c space with a sharp base is
metrizable \cite{A2000}, it follows that $G, bG$ and $Y$ are separable and metrizable
spaces by Theorem~\ref{t0}.

Case 2: The space $Y$ is Lindel$\ddot{\mbox{o}}$f.

It follows from Lemma~\ref{l1} that $G$ is a topological group.
Since $Y$ has a sharp base, the space $Y$ has a point-countable
base \cite{A2000}. Therefore, we have $G, bG$ and $Y$ are separable and
metrizable spaces \cite{A2007}.
\end{proof}

Next, we consider two questions posed in \cite{Liu2009} and
\cite{Liu20091}, respectively.

\begin{question}\cite[Question 6]{Liu2009}\label{q0}
Let $G$ be a non-locally compact topological group, if the remainder
$Y=bG\backslash G$ is a quotient $s$-image of a metric space, are
$G$ and $bG$ separable and metrizable?
\end{question}

\begin{question}\cite[Question 5.2]{Liu20091}\label{q1}
Let $G$ be a non-locally compact topological group, if the remainder
$Y=bG\backslash G$ of a Hausdorff compactification of $G$ has a
point-countable weak base, are $G$ and $bG$ separable and
metrizable?
\end{question}

Now we give a partial answers to for Questions~\ref{q0}
and~\ref{q1}.

\begin{lemma}\label{l10}
Let $G$ be a $k$-gentle paratopological group. Then $G$ is
Lindel$\ddot{\mbox{o}}$f if and only if there exists a
compactification $bG$ such that, for any compact subset $F\subset
Y=bG\setminus G$, there is a compact subset $L\subset
Y$ which contains $F$ and is a $G_{\delta}$-subset in $Y$.
\end{lemma}

\begin{proof}
If $G$ is Lindel$\ddot{\mbox{o}}$f, then $Y$ is of countable type by
Theorem~\ref{t14}. Therefore, we only need to show the
sufficiency. It follows from Lemma~\ref{l0} that $Y$ is
pseudocompact or Lindel$\ddot{\mbox{o}}$f.

Case 1: The space $Y$ is Lindel$\ddot{\mbox{o}}$f.

It follows from Lemma~\ref{l1} that $G$ is a topological group, and
hence $G$ is a paracompact $p$-space. Therefore, the space $G$ is
Lindel$\ddot{\mbox{o}}$f by \cite[Lemma 2.3]{A20091}.

Case 2: The space $Y$ is pseudocompact.

Let $F\subset
Y$ be a compact subset. Then there is a compact subset $L\subset
Y$ which contains $F$ and is a $G_{\delta}$-subset in $Y$. Since compact subset $L$ of $Y$ is a $G_{\delta}$-set and $Y$ is
pseudocompact, it is well known that compact subset $L\subset
Y$ has a countably open neighborhood base, and hence $Y$ is of
countable type. Therefore $G$ is Lindel$\ddot{\mbox{o}}$f by Theorem~\ref{t14}.
\end{proof}

\begin{theorem}\label{t1}
Suppose that $G$ is a non-locally compact, $k$-gentle
paratopological group, and $Y=bG\setminus G$ is a remainder of $G$.
Then the following conditions are equivalent:
\begin{enumerate}
\item the space $Y$ is of subcountable type\footnote{Recall that a space $X$ is of {\it subcountable type} \cite{A20091} if every
compact subspace $F$ of $X$ is contained in a compact $G_{\delta}$
subspace $K$ of $X$.};

\item the space $Y$ is of countable type.
\end{enumerate}
\end{theorem}

\begin{proof}
It is easy to see by Lemma~\ref{l10} and Theorem~\ref{t14}.
\end{proof}

\begin{theorem}\label{t2}
Suppose that $G$ is a non-locally compact, $k$-gentle
paratopological group, and $Y=bG\setminus G$ is a remainder of $G$.
If $Y$ is $\kappa$-perfect\footnote{Recall that a space $X$ is of
{\it $\kappa$-perfect} \cite{A20091} if every compact subspace $F$ is
a $G_{\delta}$ subspace of $X$.}, then $Y$ is first countable.
\end{theorem}

\begin{proof}
It follows from Theorem~\ref{t1} that $Y$ is of countable type.
Since every point of $Y$ is a $G_{\delta}$-point, the space $Y$ is first
countable.
\end{proof}

Let $\mathcal{A}$ be a collection of subsets of $X$. The collection $\mathcal{A}$ is a {\it
$p$-metabase} \cite{BD} for $X$ if for distinct points $x, y\in X$, there
exists an $\mathcal{F}\in \mathcal{A}^{<\omega}$ such that $x\in
(\cup\mathcal{F})^{\circ}\subset\cup\mathcal{F}\subset X-\{y\}$.

\begin{lemma}\label{l11}
Suppose that $X$ has a point-countable $p$-metabase.
Then each countably compact subset of $X$ is a compact, metrizable
$G_{\delta}$-subset of $X$.
\end{lemma}

\begin{proof}
Suppose that $\mathscr{U}$ is a point-countable $p$-metabase of $X$,
 and that $K$ is a compact subset of $X$. Then $K$ is compact by
\cite{BD}. It follows from \cite{GMT1} that for distinct $x, y\in
X$, there exists a finite subfamily $\mathscr{F}\subset
 \mathscr{U}$ such that $x\in (\cup\mathscr{F})^{\circ}\subset\cup\mathscr{F}\subset
 X-\{y\}$. According to a generalized Mi$\breve{s}\breve{c}$enko's Lemma in
\cite[Lemma 6]{YL}, there are only countably many minimal
neighborhood-covers\footnote{Let $\mathscr{P}$ be a collection of
subsets of $X$ and $A\subset X$. The collection $\mathscr{P}$ is a {\it
neighborhood-cover} of $A$ if $A\subset (\cup\mathscr{P})^{\circ}$.
A neighborhood-cover $\mathscr{P}$ of $A$ is a {\it minimal
neighborhood-cover} if for each $P\in \mathscr{P}$,
the family $\mathscr{P}\setminus\{P\}$ is not a neighborhood-cover of $A$.} of
$K$ by finite elements of $\mathscr{U}$, say $\{\mathscr{V}(n):n\in
\mathbb{N}\}$. For each $n\in \mathbb{N}$, let
$V(n)=\cup\mathscr{V}(n)$. Then $K\subset\cap\{V(n):n\in
\mathbb{N}\}$. Suppose that $x\in X\setminus K$. For each point
$y\in K$, there is an $\mathscr{F}_{y}\in\mathscr{U}^{<\omega}$ with
$y\in (\cup\mathscr{F}_{y})^{\circ}\subset
\cup\mathscr{F}_{y}\subset X-\{x\}$. Then there is some
sub-collection of $\cup\{\mathscr{F}_{y}:y\in K\}$ which is a
minimal finite neighborhood-covers of $K$, since $K$ is compact.
Therefore, we obtain one of the collections $\mathscr{V}(n)$ with
$K\subset V(n)=\cup\mathscr{V}(n)\subset X-\{x\}$.
\end{proof}

\begin{theorem}\cite{LFC2009}\label{t10}
Suppose that $G$ is a non-locally compact topological group, and
that $Y=bG\setminus G$ has a locally point-countable
$p$-metabase. Then $G$ and $bG$ are separable and metrizable if
$\pi$-character of $Y$ is countable.
\end{theorem}

\begin{theorem}
Suppose that $G$ is a non-locally compact topological group, and
that $Y=bG\setminus G$ has a locally point-countable
$p$-metabase. Then $G$ and $bG$ are separable and metrizable spaces.
\end{theorem}

\begin{proof}
Claim 9: The space $Y$ is $\kappa$-perfect.

Let $F$ be any compact subset in $Y$. For each $y\in Y$, there
exists an open subset $U_{y}$ such that $U_{y}$ has a point-countable
$p$-metabase. Since $F$ is compact, there exists a
finite subset $A\subset F$ such that $F\subset\bigcup_{y\in
A}U_{y}$. Without loss of generality, we denote $\{U_{y}: y\in A\}$
by $\{U_{y_{1}}, \cdots ,U_{y_{m}}\}$. Then $\{U_{y_{i}}\cap F: i=1,
\cdots , m\}$ is a relatively open cover of $F$. For each $1\leq
i\leq m$, there is a closed subset $F_{i}$ such that $F_{i}\subset
U_{y_{i}}$ and $F=\bigcup_{i=1}^{i=m}F_{i}$. For each $1\leq i\leq
m$, the set $F_{i}$ is a $G_{\delta}$-set by Lemma~\ref{l11}, and hence
there exists a sequence of open subsets $\{V_{in}\}$ of $U_{y_{i}}$
such that $F_{i}=\bigcap_{n=1}^{n=\infty}V_{in}$. Put
$W_{n}=\bigcup_{i=1}^{i=m}V_{in}$. Then
$F=\bigcap_{n=1}^{\infty}W_{n}$. In fact, it is obvious that
$F\subset\bigcap_{n=1}^{\infty}W_{n}$. We only need to show that
$\bigcap_{n=1}^{\infty}W_{n}\subset F$. Suppose not, let $x\in
\bigcap_{n=1}^{\infty}W_{n}\setminus F$. Since $x\not\in F_{i}$ for
each $1\leq i\leq m$, there exists a $k_{i}\in \mathbb{N}$ such that
$x\not\in V_{ik_{i}}$. Let $l=\mbox{max}\{k_{i}: 1\leq i\leq m\}$.
Then $x\not\in\bigcup_{i=1}^{m}V_{il}=W_{l}$, which is a
contradiction with $x\in W_{l}$.

By the Claim 9 and Theorem~\ref{t2}, $Y$ is first countable.
Therefore, $G$ and $bG$ are separable and metrizable spaces by
Theorem~\ref{t10}.
\end{proof}

Finally, we discuss the remainders of rectifiable spaces.

\begin{lemma}\cite{A2009}\label{l14}
Suppose that $G$ is a paracompact rectifiable space, and
$Y=bG\setminus G$ has a $G_{\delta}$-diagonal if and only if $Y$,
$G$ and $bG$ are separable and metrizable spaces.
\end{lemma}

\begin{theorem}
Suppose that $G$ is a paracompact rectifiable space, and
$Y=bG\setminus G$ has locally a $G_{\delta}$-diagonal if and only if
$Y$, $G$ and $bG$ are separable and metrizable spaces.
\end{theorem}

\begin{proof}
Claim 10: The space $Y$ is Lindel$\ddot{\mbox{o}}$f.

We can assume that $G$ is non-locally compact, and otherwise, the space $Y$ is
compact.

Case 1: The space $Y$ is countably compact.

Since $Y$ has locally a $G_{\delta}$-diagonal, the space $Y$ is locally
metrizable, and hence $Y$ is locally compact, which is a
contradiction with $Y$ is nowhere locally compact.

Case 2: The space $Y$ is non-countably compact.

By \cite[Theorem 3.1]{A2009}, the space $Y$ is Lindel$\ddot{\mbox{o}}$f or
pseudocompact. So we assume that $Y$ is pseudocomapact. Since each
point in $Y$ is a $G_{\delta}$-point, the space $Y$ is first countable. Since
$Y$ is non-countably compact, it follows, by a standard argument,
that $G$ has a countable $\pi$-base at some point which is an
accumulation point of some countable subset of $Y$. Hence, the space $G$ is
metrizable \cite{G1996}. Therefore, the space $Y$ is Lindel$\ddot{\mbox{o}}$f,
since $G$ is of countable type.

It follows from \cite[Lemma 11]{Liu2009} and the Claim 10 that $Y$
has a $G_{\delta}$-diagonal, and therefore, it is easy to see that
the theorem is verified by Lemma~\ref{l14}.
\end{proof}

\begin{theorem}
Suppose that $G$ is a rectifiable space, and $Y=bG\setminus G$ is a
remainder of $G$. If $Y$ has countable pseudocharacter, then at
least one of the following conditions satisfies
\begin{enumerate}
\item the space $G$ is a strong $p$-space;

\item the space $Y$ is first countable.
\end{enumerate}
\end{theorem}

\begin{proof}
One can prove in a similar way in Theorem~\ref{t3} that $G$ is of
countable type or $Y$ is first countable. By \cite[Corollary
2.8]{A2009}, it follows that $G$ is a strong $p$-space or $Y$ is a first countable
space.
\end{proof}

The following question is still open.

\begin{question}
Let $G$ be a topological group. If $G$ has a first countable
remainder, is $G$ metrizable?
\end{question}

It is well known that there exists a non-metrizable paratopological group with a
first countable remainder. In fact, Alexandorff's
double-arrow space is a Hausdorff compactification of Sorgenfrey
line, its remiander is still a copy of Sorgenfrey line. However, we
have the following question.

\begin{question}
Let $G$ be a rectifiable space. If $G$ has a first countable
remainder, is $G$ metrizable?
\end{question}

Next we give some partial answers to this question.

The proofs of the following two theorem are identical to the proofs
of Theorem 2.1 and 2.2 in \cite{A20071}, respectively.

\begin{theorem}
Let $G$ be a non-compact rectifiable space such that $G^{\omega}$
has a first countable remainder. Then $G$ is metrizable.
\end{theorem}

\begin{theorem}
A rectifiable space $G$ is metrizable if there is a nowhere locally
compact metrizable space (or first countable space) $M$ such that
the product space $G\times M$ has a first countable remainder.
\end{theorem}

\section{Open problems}
Here, we list some open problems about rectifiable spaces, which
mainly appear in \cite{A2009} and \cite{A2008}.

\begin{question}\cite[Open problem 5.7.6]{A2008}
Suppose that $G$ is a (regular, Tychonoff) paratopological group
which is also a rectifiable space. Is $G$ homeomorphic to a
topological group?
\end{question}

\begin{question}\cite[Open problem 5.7.7]{A2008}
Is every regular rectifiable space Tychonoff?
\end{question}

\begin{question}\cite[Open problem 5.7.8]{A2008}
Is every regular rectifiable space of countable pseudocharacter
submetrizable? Is it Tychonoff?
\end{question}

\begin{question}\cite[Problem 5.9]{A2009}
Is every rectifiable $p$-space paracompact? What if the space is
locally compact?
\end{question}

\begin{question}\cite[Problem 5.10]{A2009}
Is every rectifiable $p$-space with a countable Souslin number
Lindel$\ddot{\mbox{o}}$f? What if we assume the space to be
separable? Separable and locally compact?
\end{question}

\begin{question}\cite[Problem 5.11]{A2009}
Is every rectifiable $p$-space a $D$-space?
\end{question}

\begin{question}
Is every sequential rectifiable space with a point-countable
$k$-network a paracompact space?
\end{question}

It is easy to see that Corollary~\ref{c0} gives a partial answer to
this question.

\begin{question}
Let $G$ be a rectifiable. If $F, P$ are compact and closed subsets
of $G$ respectively,  is $P\cdot F$ or $F\cdot P$ closed in $G$?
\end{question}

Obviously, if both $F, P$ are compact subset of $G$, then $P\cdot F$
and $F\cdot P$ are compact in $G$ since $p$ is a continuous map from
$G\times G$ onto $G$.

\bigskip

{\bf Acknowledgements}. We wish to thank
the referee for the detailed list of corrections, suggestions to the paper, and all her/his efforts
in order to improve the paper.

\end{document}